\documentclass[12pt]{article}
\usepackage{graphicx}
\usepackage{amsmath}
\usepackage{amssymb}
\usepackage{theorem}
\usepackage{changebar, xcolor}

\usepackage{hyperref}

\numberwithin{equation}{section}

\newtheorem{thm}{Theorem}[section]
\newtheorem{lemma}[thm]{Lemma}
\newtheorem{prop}[thm]{Proposition}
\newtheorem{cor}[thm]{Corollary}
{\theorembodyfont{\rmfamily}
\newtheorem{defn}[thm]{Definition}
\newtheorem{example}[thm]{Example}

\newtheorem{rmk}[thm]{Remark}
}

\newcommand{\qed}{\hfill \mbox{\raggedright \rule{.07in}{.1in}}}
 
\newenvironment{proof}{\vspace{1ex}\noindent{\bf
Proof}\hspace{0.5em}}{\hfill\qed\vspace{1ex}}
\newenvironment{pfof}[1]{\vspace{1ex}\noindent{\bf Proof of
#1}\hspace{0.5em}}{\hfill\qed\vspace{1ex}}

\newcommand{\R}{{\mathbb R}}
\newcommand{\C}{{\mathbb C}}
\newcommand{\Z}{{\mathbb Z}}
 \newcommand{\E}{{\mathbb E}}
 \newcommand{\T}{\mathbb T}
 \newcommand{\bbS}{\mathbb S}
 \newcommand{\BBW}{\mathbb W}

\newcommand{\cB}{{\mathcal B}}
\newcommand{\cF}{{\mathcal F}}
\newcommand{\cW}{{\mathcal W}}
\newcommand{\cS}{\mathcal{S}}
\newcommand{\cK}{\mathcal{K}}
\newcommand{\cA}{\mathcal{A}}
\newcommand{\cL}{\mathcal{L}}
\newcommand{\cP}{\mathcal{P}}

\newcommand{\vf}{\varphi}
\newcommand{\ve}{\varepsilon}

\newcommand{\bmu}{\bar{\mu}}

\newcommand{\eps}{\epsilon}
\newcommand{\sgn}{\operatorname{sgn}}
\newcommand{\diam}{\operatorname{diam}}

\renewcommand{\Re}{\operatorname{Re}}

\newcommand{\SMALL}{\textstyle}

\title{Martingale approximations and  anisotropic Banach spaces with an application to the time-one map of a Lorentz gas}

\author{
Mark Demers\thanks{Department of Mathematics, Fairfield University, Fairfield, CT 06824, USA.  Email: mdemers@fairfield.edu}
 \and Ian Melbourne\thanks{Mathematics Institute, University of Warwick, Coventry, CV4 7AL, UK.  Email: i.melbourne@warwick.ac.uk}
 \and Matthew Nicol\thanks{Department of Mathematics, University of Houston, Houston TX 77204-3008, USA.  Email: nicol@math.uh.edu}
}

\date{30 December 2018.  Revised 13 February 2020}

\begin{document}

\maketitle

\begin{abstract}
In this paper, we show how the Gordin martingale approximation method fits into the anisotropic Banach space framework. In particular, for the time-one map of a finite horizon planar periodic Lorentz gas, we prove that H\"older observables satisfy statistical limit laws such as the central limit theorem and associated invariance principles.
Previously, these properties were known only for a restricted class of observables, excluding for instance velocity.
\end{abstract}

\section{Introduction}

The traditional approach to proving decay of correlations and statistical limit laws for deterministic dynamical systems, following~\cite{Bowen75, Ruelle78,Sinai72}
and continuing with Young~\cite{Young98,Young99},
involves symbolic coding.  In particular, by quotienting along stable leaves one passes from an invertible dynamical system to
a one-sided shift.  Decay of correlations is then a consequence of the contracting properties of the associated transfer operator.  In addition, Nagaev perturbation arguments~\cite{Gouezel15,HennionHerve} and the martingale approximation method of Gordin~\cite{Gordin69} are available in this setting, leading to numerous statistical limit laws.  These results on decay of correlations and statistical limit laws are then readily passed back to the original dynamical system.

A downside to this approach is that geometric and smooth structures associated to the underlying dynamical system are typically destroyed by symbolic coding.
In recent years, a method proposed by~\cite{BlankKellerLiverani} and developed extensively by numerous authors (for recent articles with up-to-date references see~\cite{Baladi17,Demers18}) uses anisotropic Banach spaces of distributions to study the underlying dynamical system directly.  In particular, the method does not involve quotienting along stable manifolds.  
This leads to results on rates of decay of correlations  and also to various statistical limit laws via Nagaev perturbation arguments, see especially Gou\"ezel~\cite{Gouezel10}.

However, so far
Gordin's martingale approximation argument has been absent from the anisotropic Banach space framework.
This is the topic of the current paper.
The utility of such an approach is illustrated by the following example.

\begin{example} \label{ex:gas}
The landmark result of Young~\cite{Young98} established exponential decay of correlations for 
the collision map corresponding to planar periodic dispersing billiards with finite horizon.  
The method, which involves symbolic coding, also yields the central limit theorem (CLT) for H\"older observables, recovering results of~\cite{BunimovichSinaiChernov91}.

Turning to the corresponding flow, 
known as the finite horizon planar periodic Lorentz gas, the CLT follows straightforwardly from the result for billiards~\cite{BunimovichSinaiChernov91,MT04}.  However, decay of correlations for the Lorentz gas and the CLT for the time-one map of the Lorentz gas are much harder.
Superpolynomial decay of correlations was established for sufficiently regular observables in~\cite{M07} (see also~\cite{M18}) using symbolic coding and Dolgopyat-type estimates~\cite{Dolgopyat98b}.  This method also yields the CLT for the time-one map~\cite{AMV15,MT02}, but again only for sufficiently regular observables.
Here, ``regular'' means smooth along the flow direction, so this excludes many physically relevant observables such as velocity.
The rate of decay of correlations 
 was improved to subexponential decay~\cite{Chernov07} and finally in a recent major breakthrough to exponential decay~\cite{BaladiDemersLiverani18}.
Both references handle H\"older observables,  suggesting that statistical limit laws such as the CLT for the time-one map should hold for general H\"older observables.  

Currently the Nagaev method is unavailable for Lorentz gases, and as a consequence the CLT for the time-one map was previously unavailable except for a restricted class of observables.  We show that the Gordin approach is applicable and hence the CLT and related limit laws are indeed satisfied by H\"older observables for these examples.
In particular, observables such  as velocity are covered for the first time.  
\end{example}

In the remainder of the introduction, we describe some of the limit laws that follow from the methods in this paper.
For definiteness, we focus on Example~\ref{ex:gas}.
Let $X$ be the three-dimensional phase space corresponding to a finite horizon planar periodic Lorentz gas, with invariant volume measure $\mu$,
and let $T:X\to\ X$ be the time-one map of the Lorentz flow.  Let $\phi:X\to\R$ be a H\"older observable with mean zero and define the Birkhoff sum $\phi_n=\sum_{j=0}^{n-1} \phi \circ T^j$.  
It follows from~\cite{BaladiDemersLiverani18,Chernov07} that
we can define 
\[
\sigma^2=\lim_{n\to\infty} n^{-1}\int_X \phi_n^2\,d\mu=\sum_{n=-\infty}^\infty
\int_X\phi\,\phi\circ T^n\,d\mu.
\]
By~\cite[Theorem~B and Remark~1.1]{AMV15}, typically $\sigma^2>0$ (the case $\sigma^2=0$ is of infinite codimension).
We obtain the following results.\footnote{In what follows, $\to_d$ denotes convergence in distribution
while $\to_w$ denotes weak convergence.}

\vspace{1ex}
\noindent{\bf CLT:}
$n^{-1/2}\phi_n\to _d N(0,\sigma^2)$ as $n\to\infty$.  That is
\[
\lim_{n\to\infty} \mu(x \in X:n^{-1/2}\phi_n(x)\le c)=(2\pi\sigma^2)^{-1/2}\int_{-\infty}^c e^{-y^2/(2\sigma^2)}\,dy
\quad\text{for all $c\in\R$.}
\]

\vspace{1ex}
\noindent{\bf Weak invariance principle (WIP):}
Define $W_n(t)=
n^{-1/2}\phi_{nt}$ for $t=0,\frac1n,\frac2n,\dots$ and linearly interpolate to obtain $W_n\in C[0,1]$.  
Then $W_n\to_w W$ where
$W$ denotes Brownian motion with variance $\sigma^2$. 

\vspace{1ex}
\noindent{\bf Moment estimates:} 
For every $p\ge1$ there exists $C_p>0$ such that
$|\phi_n|_p\le C_pn^{1/2}$.  Consequently, $\lim_{n\to\infty}n^{-p/2}\|\phi_n\|_p^p=\E|Y|^p$ where $Y=_d N(0,\sigma^2)$.

\vspace{1ex}
\noindent{\bf Homogenization:}
Now suppose that $\phi:X\to\R^k$.  We continue to suppose that
$\phi$ is $C^\eta$ for some $\eta\in(0,1]$ and that $\int_X \phi\,d\mu=0$.
Consider the fast-slow system
\begin{align} \nonumber
 x(n+1) & =x(n)+\eps^2a(x(n))+\eps b(x(n))\phi(y(n)), \\
y(n+1) & =Ty(n), \label{eq:fs}
\end{align}
where $x(0)=\xi\in\R^d$
 and $y(0)$ is drawn randomly from $(X,\mu)$.
We suppose that 
$a:\R^d\to\R^d$ lies in $C^{1+\eta}$ and
$b:\R^d\to\R^{d\times k}$ lies in $C^{2+\eta}$.
Solve~\eqref{eq:fs} to obtain
\[
x_\eps(n)=\xi+\eps^2\sum_{j=0}^{n-1}a(x_\eps(j))+\eps\sum_{j=0}^{n-1}b(x_\eps(j))\phi(y(j)), \qquad y(n)=T^ny(0),
\]
and let $\hat x_\eps(t)=x_\eps([t/\eps^2])$.  This defines a random process on the probability space $(X,\mu)$ depending on $y(0)\in X$.
Then $\hat x_\eps\to_w Z$ as $\eps\to0$, where $Z$ satisfies an It\^{o} stochastic differential equation $dZ=\tilde a(Z)dt+b(Z)\,dW$, $Z(0)=\xi$,
where $W$ is a $k$-dimensional Brownian motion with covariance matrix $\Sigma$ and 
\begin{align} \label{eq:drift}
\tilde a(x)=a(x)+\sum_{\alpha=1}^d\sum_{\beta,\gamma=1}^k E^{\gamma\beta}\frac{\partial b^\beta}{\partial x_\alpha}(x)b^{\alpha\gamma}(x).
\end{align}
Here, $b^\beta$ is the $\beta$'th column of $b$ and the matrices $\Sigma,E\in\R^{k\times k}$ are given by
\[
\Sigma^{\beta\gamma}  = \sum_{n=-\infty}^\infty \int_X \phi^\beta\,\phi^\gamma\circ T^n\,d\mu, \qquad
E^{\beta\gamma}  = \sum_{n=1}^\infty \int_X\phi^\beta\,\phi^\gamma\circ T^n\,d\mu.
\]

\vspace{2ex}
The remainder of this paper is organized as follows.
In Section~\ref{sec:mart}, we recall background material on martingale-coboundary decompositions and statistical limit laws.
In Section~\ref{sec:main}, we state an 
abstract theorem on obtaining martingale-coboundary decompositions for invertible systems with stable directions.
In Section~\ref{sec:gas}, we apply our results to the time-one map of the Lorentz gas.

\section{Martingale approximations}
\label{sec:mart}

In this section, we review the approach going back to Gordin~\cite{Gordin69}.
This method yields martingale approximations for observables of dynamical systems leading to various limit theorems.  
Related references include~\cite{BalintGouezel06,BM18,DedeckerMerlevedePene13,DedeckerRio00,Heyde75, Liverani96, Viana,Volny93,Volny07}.
Let $(X,\mu)$ be a probability space, and let
$T: X\to X$ be an invertible
ergodic measure-preserving transformation.
Let $\cF_0$ be a sub-$\sigma$-algebra of the underlying $\sigma$-algebra on $X$
such that $T^{-1}\cF_0\subseteq\cF_0$.
Consider an observable\footnote{Most observables in this paper are real-valued, but occasionally in this section we consider observables with values in $\R^k$.  We write $L^1(X,\R^k)$ to denote vector-valued observables and write $L^1(X)$ instead of $L^1(X,\R)$.}
 $\phi\in L^1(X)$ with $\int_X\phi\,d\mu=0$.

\begin{defn} \label{def:mart}
We say that 
$\phi$ admits a {\em martingale-coboundary decomposition} if
\[
\phi=m+\chi\circ T-\chi,
\]
where $m,\chi\in L^1(X)$, 
$m$ is $\cF_0$-measurable, and $\E[m|T^{-1}\cF_0]=0$.
\end{defn}

The conditions on $m$ in Definition~\ref{def:mart} mean that $\{m\circ T^{-n}:n\in\Z\}$ is a sequence of martingale differences with respect to the 
filtration $\{T^n\cF_0:n\in\Z\}$.  

\begin{prop}  \label{prop:mart}
Let $\phi\in L^p(X)$ for some $p\ge1$.  Suppose that
\begin{equation} \label{eq:mart}
 \SMALL  
 \sum_{n\ge 1} |\E[\phi\circ T^{-n}|\cF_0]|_p<\infty, \qquad
\sum_{n\ge0} \big|\E[\phi\circ T^n|\cF_0]-\phi\circ T^n\big|_p<\infty.
\end{equation}
Then $\phi$ admits a martingale-coboundary decomposition with $m,\,\chi\in L^p(X)$.
\end{prop}

\begin{proof}  
This is a standard argument~\cite{Heyde75,Volny93}.  We give the details for completeness.
By~\eqref{eq:mart},
\[
\SMALL \chi=\sum_{n\ge 0} (\E[\phi\circ T^n|\cF_0]-\phi\circ T^n)
+\sum_{n\ge 1} \E[\phi\circ T^{-n}|\cF_0]
\]
converges in $L^p(X)$.
Define
$m=\phi+\chi-\chi\circ T\in L^p(X)$.
Then
\begin{equation} \label{eq:m}
\SMALL m=\sum_{n=-\infty}^\infty (g_n-g_n\circ T)
=\sum_{n=-\infty}^\infty (g_{n+1}-g_n\circ T),
\end{equation}
where $g_n=\E[\phi\circ T^n|\cF_0]$.  

Clearly, $g_n=\E[\phi\circ T^n|\cF_0]$ is $\cF_0$-measurable.  Also,
$g_n\circ T$ is measurable with respect to $T^{-1}\cF_0\subseteq \cF_0$.
Hence $m$ is $\cF_0$-measurable.

Next, note that
$g_n\circ T=\E[\phi\circ T^n|\cF_0]\circ T=
\E[\phi\circ T^{n+1}|T^{-1}\cF_0]$.  
Hence
\[
\E[g_n\circ T|T^{-1}\cF_0]=\E[\phi\circ T^{n+1}|T^{-1}\cF_0]
=\E[\E[\phi\circ T^{n+1}|\cF_0]|T^{-1}\cF_0]
=\E[g_{n+1}|T^{-1}\cF_0],
\]
where we used that $T^{-1}\cF_0\subseteq \cF_0$.
Substituting into~\eqref{eq:m},
we obtain $\E[m|T^{-1}\cF_0]=0$ as required.
\end{proof}


\subsection*{Central limit theorem and invariance principles}

\begin{cor} \label{cor:mart}
Assume that $\phi\in L^2(X)$ and conditions~\eqref{eq:mart} hold with $p=2$.
Then
the CLT and WIP hold with $\sigma^2=\int_X m^2\,d\mu=\lim_{n\to\infty}n^{-1}|\phi_n|_2^2$.
\end{cor}

\begin{proof}  This is a 
standard application of martingale limit theorems~\cite{Gordin69}.
\end{proof}

Somewhat surprisingly, by the results of~\cite{DedeckerRio00}, if $\phi\in L^\infty(X)$ and conditions~\eqref{eq:mart} hold for $p=1$, then automatically $m\in L^2(X)$ even though Proposition~\ref{prop:mart} only gives $m,\chi\in L^1(X)$.
This suffices for the CLT.
Related references for this phenomenon whereby $m$ has extra regularity include~\cite{KipnisVaradhan86,Liverani96,MaxwellWoodroofe00,PeligradUtev05,TyranKaminska05,Volny07}.
In particular, the following result holds:

\begin{thm} \label{thm:CLT}
Assume that $\phi\in L^\infty(X)$ and conditions~\eqref{eq:mart} hold with $p=1$.  Then the CLT and WIP hold.
\end{thm}

\begin{proof}
The CLT and WIP in reverse time (as $n\to-\infty$) is an immediate consequence 
of~\cite[Corollary~4]{DedeckerRio00}.
Passing from reverse time to forward time is standard
(see for example~\cite[Section~4.2]{KM16}).
\end{proof}

Now let $\phi$ be vector-valued with values in $\R^k$.
Define c\`adl\`ag processes $W_n$ in $\R^k$ and $\BBW_n\in\R^{k\times k}$:
\[
W_n(t)=n^{-1/2}\sum_{0\le j<nt}\phi\circ T^j, \qquad
\BBW_n^{\beta\gamma}(t)=n^{-1}\sum_{0\le i<j<nt}\phi^\beta \circ T^i \phi^\gamma \circ T^j.
\]

\begin{prop}[Iterated WIP] \label{prop:iwip}
Suppose that $T$ is mixing.
Assume that $\phi\in L^2(X,\R^k)$ and conditions~\eqref{eq:mart} hold with $p=2$.  Then 
\begin{itemize}
\item[(i)] The series
\(
\Sigma^{\beta\gamma}  = \sum_{n=-\infty}^\infty \int_X \phi^\beta\,\phi^\gamma\circ T^n\,d\mu, \,
E^{\beta\gamma}  = \sum_{n=1}^\infty \int_X\phi^\beta\,\phi^\gamma\circ T^n\,d\mu,
\)
converge.
\item[(ii)] 
$(W_n,\BBW_n)\to_w (W,\BBW)$,
where $W$ is a $k$-dimensional Brownian motion with covariance matrix $\Sigma$ and $\BBW^{\beta\gamma}(t)=\int_0^t W^\beta\,dW^\gamma+E^{\beta\gamma}t$.
\end{itemize}
\end{prop}

\begin{proof}
By Proposition~\ref{prop:mart}, $\phi$ admits a martingale-coboundary decomposition with $m$, $\chi\in L^2(X,\R^k)$, so the result holds
by~\cite[Theorem~4.3]{KM16}.
\end{proof}

\subsection*{Moments}

For optimal moment estimates, the following projective version of conditions~\eqref{eq:mart} are better suited:
\begin{equation} \label{eq:proj}
 \SMALL  
 \sum_{n\ge 1} n^{-1/2}|\E[\phi\circ T^{-n}|\cF_0]|_p<\infty, \qquad
\sum_{n\ge0} n^{-1/2}\big|\E[\phi\circ T^n|\cF_0]-\phi\circ T^n\big|_p<\infty.
\end{equation}

\begin{prop} \label{prop:mom} 
Assume $\phi\in L^p(X)$ and conditions~\eqref{eq:proj} hold for some $p>2$.
Then
$\big|\max_{k\le n}|\phi_k|\big|_p=O(n^{1/2})$.

If in addition $n^{-1/2}\phi_n\to_d Y$ for some $L^p$ random variable $Y$,
then
$\lim_{n\to\infty}n^{-q/2}|\phi_n|_q^q=\E|Y|^q$ 
for all $q<p$.
\end{prop}

\begin{proof}
Let $A_n=\sum_{j=1}^n \phi\circ T^{-j}$.
Then\footnote{We use the notation $A \ll B$ to denote $A \le \mbox{const.} B$, where the constant is independent
of the other parameters present.} for $r\ge1$,
\begin{align*}
\sum_{k=0}^{r-1}2^{-k/2} & |\E(A_{2^k}|\cF_0)|_p
\le \sum_{k=0}^{r-1}2^{-k/2}\sum_{j=1}^{2^k}|\E(\phi\circ T^{-j}|\cF_0)|_p
\\ &=\sum_{j=1}^{2^{r-1}}
\sum_{k=\lceil \log_2j \rceil}^{r-1}2^{-k/2}|\E(\phi\circ T^{-j}|\cF_0)|_p
  \ll \sum_{j=1}^{2^{r-1}}
j^{-1/2}|\E(\phi\circ T^{-j}|\cF_0)|_p.
\end{align*}
By condition~\eqref{eq:proj}, $\sum_{k=0}^\infty 2^{-k/2}|\E(A_{2^k}|\cF_0)|_p<\infty$.
Similarly, 
$\sum_{k=1}^{\infty}2^{-k/2}|A_{2^k}-\E(A_{2^k}|T^{2^k}\cF_0)|_p<\infty$.
Recalling that $T^{-1}\cF_0\subseteq\cF_0$, it follows from~\cite[Corollary~3.9]{DedeckerMerlevedePene13}
that $\big|\max_{k\le 2^r}|A_k|\big|_p \ll 2^{r/2}$.

For general $n\ge1$ choose $r\ge1$ so that $2^{r-1}<n\le 2^r$.
Then
\[
\big|\max_{k\le n}|A_k|\big|_p\le
\big|\max_{k\le 2^r}|A_k|\big|_p\ll 
2^{r/2}\le (2n)^{1/2}.
\]
Finally,
$\phi_k=(A_n-A_{n-k})\circ T^n$
so
\[
\big|\max_{k\le n}|\phi_k|\big|_p=
\big|\max_{k\le n}|A_n-A_{n-k}|\big|_p
\le
2\big|\max_{k\le n}|A_k|\big|_p\ll n^{1/2},
\]
proving the first statement.

The second statement is an immediate consequence of the first, see for example~\cite[Lemma~2.1(e)]{MTorok12}.
\end{proof}

Now let $\phi$ be vector-valued with values in $\R^k$ and
define $S_n^{\beta\gamma}=\sum_{0\le i<j<n}\phi^\beta\circ T^i\,\phi^\gamma\circ T^j$.

\begin{prop} \label{prop:imom}
Assume that $\phi\in L^p(X,\R^k)$ and conditions~\eqref{eq:mart} hold for some $p\ge4$.  Then 
$\big|\max_{k\le n}|S_k^{\beta\gamma}|\big|_{p/2}=O(n)$.
\end{prop}

\begin{proof}
By Proposition~\ref{prop:mart}, we have a martingale-coboundary decomposition
$\phi=m+\chi\circ T-\chi$ with $m,\,\chi\in L^p(X,\R^k)$.
Write
\begin{align*}
S_n^{\beta\gamma} & =
\sum_{0\le i<j<n}m^\beta\circ T^i\, \phi^\gamma\circ T^j+
\sum_{1\le j<n}(\chi^\beta\circ T^j-\chi^\beta) \phi^\gamma\circ T^j
=I_n+J_n
\end{align*}
where
$I_n=\sum_{0\le i<j<n}m^\beta\circ T^i\, m^\gamma\circ T^j$ and
\[
J_n=
\sum_{0\le i<n-1}m^\beta\circ T^i\, (\chi^\gamma\circ T^n-\chi^\gamma\circ T^{i+1})
+
\sum_{1\le j<n}(\chi^\beta\circ T^j-\chi^\beta) \phi^\gamma\circ T^j.
\]

Now,
\[
\max_{k\le n}|J_k|\le
\sum_{0\le i<n-1}|m^\beta|\circ T^i\, (|\chi^\gamma|\circ T^n+|\chi^\gamma|\circ T^{i+1})
+
\sum_{1\le j<n}(|\chi^\beta|\circ T^j+|\chi^\beta|) |\phi^\gamma|\circ T^j.
\]
Hence
$\big|\max_{k\le n}|J_k|\big|_{p/2}\le 2n\big(|m^\beta|_p|\chi^\gamma|_p
+ |\chi^\beta|_p|\phi^\gamma|_p\big)$.

Next, we recall the identity
\[
I_k=I_n-I_{n-k}\circ T^k-(m^\beta_n-m^\beta_{n-k}\circ T^k)(m^\gamma_{n-k}\circ T^k), \quad 0\le k\le n,
\]
where $m_n^\beta=\sum_{i=0}^{n-1}m^\beta\circ T^i$.
Set 
\[
m_n^{\beta,-}=\sum_{1\le i\le n}m^\beta\circ T^{-i}, \qquad
I_n^-=\sum_{1\le j<i\le n}m^\beta\circ T^{-i}m^\gamma\circ T^{-j}.
\]
Then $m_{n-k}^\beta\circ T^k=m_{n-k}^{\beta,-}\circ T^n$ and
$I_{n-k}\circ T^k=I_{n-k}^-\circ T^n$ for all $k\le n$.
Hence
\[
I_k=\Big(I_n^- -I_{n-k}^- - (m^{\beta,-}_n-m^{\beta,-}_{n-k})m^{\gamma,-}_{n-k}\Big)\circ T^n
\]
and so
\begin{equation} \label{eq:minus}
\big|\max_{k\le n}|I_k|\big|_{p/2}
\le 2\big|\max_{k\le n}|I_k^-|\big|_{p/2}
+ 2 
\big|\max_{k\le n}|m^{\beta,-}_k|\big|_p
\big|\max_{k\le n}|m^{\gamma,-}_k|\big|_p.
\end{equation}

Now
\[
I_k^-= \sum_{i=2}^{k}X_{i}
\quad\text{where}\quad
X_{i}=m^\beta\circ T^{-i}\Big(\sum_{j=1}^{i-1}m^\gamma\circ T^{-j}\Big)
=m^\beta\circ T^{-i}m^{\gamma,-}_{i-1}.
\]
Since $\{m\circ T^{-n};\,n\ge0\}$ is a sequence of $L^p$ martingale differences, 
$\{X_{i};\,i\ge1\}$ is a sequence of $L^{p/2}$ martingale differences.
By the inequalities of Doob and Burkholder~\cite{Burkholder73},
\[
\big|\max_{k\le n}|I_k^-|\big|_{p/2}^2\ll 
|({\SMALL\sum}_{i=1}^nX_{i}^2)^{1/2}|_{p/2}^2
=|{\SMALL\sum}_{i=1}^n X_{i}^2|_{p/4}.
\]
(The implied constant depends only on $p$ and is in particular independent of $n$.) 
Hence, using that $p\ge4$,
\[
\big|\max_{k\le n}|I_k^-|\big|_{p/2}^2\ll 
{\SMALL\sum}_{i=1}^n |X_{i}^2|_{p/4}
= {\SMALL\sum}_{i=1}^n|X_{i}|_{p/2}^2
\le |m^\beta|_p^2 {\SMALL\sum}_{i=1}^n
 |m^{\gamma,-}_{i-1}|_p^2.
\]
Applying Burkholder once more, $\big|\max_{k\le n}|m^{\gamma,-}_k|\big|_p\ll n^{1/2}|m^\gamma|_p\,$; in particular
$\big|\max_{k\le n}|I_k^-|\big|_{p/2}\ll n|m^\beta|_p|m^\gamma|_p$.  
Substituting these estimates into~\eqref{eq:minus}
yields $\big|\max_{k\le n}|I_k|\big|_{p/2}\ll n|m^\beta|_p|m^\gamma|_p$ and
the result follows.
\end{proof}

\begin{rmk} \label{rmk:error}
 There is an error in~\cite[Proposition~7.1]{KM16} due to 
an inaccurate application of a (correct) result of~\cite{MN08}.  (The argument in~\cite{KM16} is fine for nonuniformly expanding maps but false for nonuniformly hyperbolic maps since the observable $\phi$ is not adapted to the filtration for the martingale.)  

This error was repeated in the first version of the current paper and was spotted by the referee.
As pointed out to us by the referee, the reference~\cite{DedeckerMerlevedePene13} can be used for the ordinary moments $\phi_n$ and this argument is now employed in the proof of Proposition~\ref{prop:mom}.
(Indeed, Proposition~\ref{prop:mom} is an improvement on the previous result~\cite[Eq.~(3.1)]{MN08} since it is no longer required that $\phi\in L^\infty(X)$.)  However, it remains an interesting open problem to obtain optimal control of the iterated moments $S_n$.
\end{rmk}

\paragraph{Homogenization}

As shown in~\cite{KM16,KM17}, rough path theory yields 
homogenization of fast-slow systems~\eqref{eq:fs} provided the iterated WIP
and suitable iterated moment estimates hold.
The iterated moment estimates have been relaxed in~\cite{CFKMZ,CFKMZprep}.
We now apply these results to the 
fast-slow system~\eqref{eq:fs}.

 Define the c\`adl\`ag process
$\hat x_\eps$ and the stochastic process $Z$ as in the introduction.  We continue to assume that
$a\in C^{1+\eta}$ and $b\in C^{2+\eta}$ for some $\eta>0$.

\begin{thm} \label{thm:fs}
Suppose that $T$ is mixing.
Assume that $\phi\in L^p(X,\R^k)$ and conditions~\eqref{eq:mart} hold with $p=4$.  Then 
$\hat x_\eps\to_w Z$ as $\eps\to0$.
\end{thm}

\begin{proof}
The iterated WIP holds by Proposition~\ref{prop:iwip}.
By~\cite[Theorem~4.10]{CFKMZ}, it now suffices to show that
$\big|\max_{k\le n}|\phi_k|\big|_{2q}=O(n^{1/2})$
and
$\big|\max_{k\le n}|S_k|\big|_q=O(n)$ for some $q>1$.
This and more follows from
Propositions~\ref{prop:mom} and~\ref{prop:imom}.
\end{proof}

\begin{rmk}  
The standard WIP and moments are insufficient to determine the limiting stochastic process $Z$.  By rough path theory~\cite{FrizHairer,Lyons98}  the iterated process $\BBW_n$ provides the extra information required to determine limiting stochastic integrals, and thereby the modified drift term~\eqref{eq:drift}.  The iterated moment estimate $S_n^{\beta\gamma}$ provides the required tightness.

Note that $\BBW_n$ and 
$S_n^{\beta\gamma}$ involve summation over $i<j$.  The behaviour of their symmetrized versions (incorporating $i>j$ terms, equivalently $i\ge j$ terms) follows immediately from the ordinary WIP and moment estimate,
and hence provides no extra information.
(Indeed the symmetrized version of $\BBW_n^{\beta\gamma}$ is $W_n^\beta W_n^\gamma$ which converges weakly to $W^\beta W^\gamma$.)
\end{rmk}
\section{Main abstract theorem}
\label{sec:main}

Let $T: X\to X$ be an invertible ergodic measure-preserving transformation on
a probability space $(X,\mu)$.
We suppose that $X$ is covered by a collection $\cW^s$ of disjoint measurable subsets, called 
``local stable leaves'', such that $TW^s(x)\subseteq W^s(Tx)$ for all
$x\in X$, where $W^s(x)$ is the partition element containing~$x$.

Let $\cF_0$ denote the $\sigma$-algebra generated by $\cW^s$.
Note that $W^s(y)\subseteq T^{-1}W^s(x)$ for all $y\in T^{-1}W^s(x)$,
so $T^{-1}W^s(x)$ is a union of elements of $\cW^s$.
Hence $T^{-1}\cF_0\subseteq \cF_0$.
We denote by $L^\infty(\cF_0)$ the set of functions in $L^\infty(X)$ that are $\cF_0$-measurable.

\begin{thm} \label{thm:main}
Let $\phi\in L^\infty(X)$ be a mean zero observable.
Assume that there exists $\beta>1$ and $C>0$ such that for all $n\ge1$,
\begin{itemize}
\item[(a)]
$|\int_X \phi\, \psi \circ T^n\, d\mu |\le C |\psi|_{\infty} n^{-\beta}$
for all $\psi\in L^\infty(\cF_0)$.
\item[(b)]  $\int_X \diam(\phi(T^nW^s))\,d\mu\le Cn^{-\beta}$.
\end{itemize}

Then the conditions in~\eqref{eq:mart} are satisfied for all $1 \le  p<\beta$,
and the conditions in~\eqref{eq:proj} are satisfied for all $1 \le  p<2\beta$.
\end{thm}

\begin{proof}
This is a standard argument.  We again give the details for completeness.

Let 
\[
\xi=|\E[\phi|T^{-n}\cF_0]|^{p-1}
\sgn \E[\phi|T^{-n}\cF_0]
=\psi\circ T^n,
\]
where
\[
\psi=|\E[\phi\circ T^{-n}|\cF_0]|^{p-1}
\sgn \E[\phi\circ T^{-n}|\cF_0]\in L^\infty(\cF_0),
\]
and $|\psi|_\infty\le|\phi|_\infty^{p-1}$.
Then
\begin{align*}
|\E[\phi\circ T^{-n}|\cF_0]|_p^p & =
|\E[\phi|T^{-n}\cF_0]|_p^p  =
 \int_X \E[\phi|T^{-n}\cF_0] \xi\,d\mu
 \\ & =\int_X \E[\phi\xi|T^{-n}\cF_0] \,d\mu
=\int_X \phi\,\xi \,d\mu =\int_X \phi\,\psi\circ T^n \,d\mu.
\end{align*}
By assumption~(a),
\[
|\E[\phi\circ T^{-n}|\cF_0]|_p^p=
\Big|\int_X \phi\,\psi\circ T^n \,d\mu\Big|
\le C|\psi|_\infty n^{-\beta}
\le C|\phi|_\infty^{p-1} n^{-\beta},
\]
and the first part of conditions~\eqref{eq:mart} and~\eqref{eq:proj} follows by taking $p$th roots
and using the restriction on $p$.

Next,
using the pointwise estimate $|\E[\phi | T^{n} \cF_0]-\phi | \le 
\diam(\phi(T^n W^s))$
and assumption~(b),
\begin{align*}
|\E[\phi\circ T^{n}|\cF_0]-\phi\circ T^n|_p^p & =
|\E[\phi|T^{n}\cF_0]-\phi|_p^p\le |\diam(\phi(T^nW^s))|_p^p
\\ & 
\le (2|\phi|_\infty)^{p-1}|\diam(\phi(T^nW^s))|_1
\le 2^{p-1}C|\phi|_\infty^{p-1}n^{-\beta}.
\end{align*}
The second part of conditions~\eqref{eq:mart} and~\eqref{eq:proj} follows.
\end{proof}

In the remainder of this section, we show that the conditions in Theorem~\ref{thm:main} are satisfied in many standard situations.  (The verifications below are not needed for our main example in Section~\ref{sec:gas}.)

\subsection{Verifying condition~(b) in Theorem~\ref{thm:main}}
\label{sec:b}

Suppose that $T:X\to X$ and $\cW^s$ are as above.
Let $Y\subseteq X$ be
a positive measure subset that is a union of local stable leaves in $\cW^s$.  Define the  first return time
$R:Y\to\Z^+$ and first return map $F:Y\to Y$, 
\[
R(y)=\inf\{n\ge1:T^ny\in Y\}, \qquad F(y)=T^{R(y)}y.
\]
   Let $h_n$ be the random variable on $X$ 
given by $h_n (x)=\# \{ 0\le j\le n: T^j x \in Y \}$.

\begin{lemma} \label{lem:b}  Let
$\phi:X\to\R$ be measurable. Suppose that 
$\mu (y\in Y: R(y)>n) =O(n^{-(\beta+1)})$ for some $\beta>1$
and that there are constants $C\ge1$, $\gamma\in(0,1)$ such that
\[
|\diam(\phi(T^n W^s))|\le C\gamma^{h_n(x)} \quad\text{for all $W^s\in\cW^s$, $n\ge1$}.
\]
Then condition~(b) in Theorem~\ref{thm:main} holds.
\end{lemma}

\begin{proof}
We have
\[
  \int_X \diam (\phi(T^n W^s))\, d\mu
 \le C  \sum_{k=0}^{n+1} \gamma^k \int_X 1_{\{h_n=k\}} d\mu
\le C\sum_{k=1}^{n+1} \gamma^k \int_Y 1_{\{h_n=k\}}R\, d\mu.
  \]
  If $y\in Y \cap \{h_n=k\}$,  then
$ \sum_{j=0}^{k-1} R\circ F^j> n$, and so
  $R\circ F^j > \frac{n}k$
  for some $j=0,\ldots, k-1$. Hence
  \[
\int_Y 1_{\{h_n=k\}}R\, d\mu
\le \sum_{j=0}^{k-1} \int_Y 1_{\{R\circ F^j\ge \frac{n}{k}\}}R\, d\mu.
  \]

  It follows from the tail assumption on $R$ that
there is a constant $C_1>0$ such that
 $\mu (y\in Y: R(y)>n) \le C_1n^{-(\beta+1)}$ and
  $\int_Y 1_{\{R>n\}}R\, d\mu \le C_1n^{-\beta}$.
Write $R=1_{\{R\le n\}}R+1_{\{R>n\}}R$.  Then
\begin{align*}
\int_Y 1_{\{R\circ F^j\ge \frac{n}{k}\}}R\, d\mu & \le
\int_Y \! n1_{\{R\circ F^j\ge \frac{n}{k}\}}d\mu
+\!\int_Y 1_{\{R>n\}}R\, d\mu
 =n\mu(R\ge {\SMALL\frac{n}{k}}) +\! \int_Y 1_{\{R>n\}}R\,d\mu
 \\ & \le C_1k^{\beta+1}n^{-\beta}+C_1n^{-\beta}
\le 2C_1k^{\beta+1}n^{-\beta}.
\end{align*}
Therefore, $\int_Y 1_{\{h_n=k\}}R\, d\mu \le 2C_1k^{\beta+2}n^{-\beta}$,
and
   \[
  \int_X \diam(\phi(T^n W^s))\,d\mu\le 2CC_1
n^{-\beta}\sum_{k=1}^\infty \gamma^k k^{\beta+2} = O(n^{-\beta}),
  \]
as required.
\end{proof}

\subsection{Verifying condition~(a) in Theorem~\ref{thm:main}}

For completeness, we show that Theorem~\ref{thm:main} includes examples that fit within the Chernov-Markarian-Zhang setup~\cite{ChernovZhang05, ChernovZhangSD,Markarian04}
(in the summable decay of correlations regime, so $\beta>1$)
for H\"older mean zero observables $\phi:X\to\R$.
In particular, we recover limit theorems that have been obtained previously for such invertible examples~\cite{KM16,MN05,MTorok12,MV16}.
Since there are no new results here, we only sketch the construction from~\cite{ChernovZhang05,Markarian04}.

\begin{rmk}  When treating examples falling within the Chernov-Markarian-Zhang setup, a significant (over)simplification is to suppose that there is exponential (or rapid) contraction of stable leaves under the underlying dynamics.  For billiards with subexponential decay of correlations, such a condition fails since on average stable directions contract as slowly as unstable directions expand.  In general, one should assume that there is an inducing set (called $Y$ below) such that expansion {\em and} contraction occurs only on visits to $Y$.  This general point of view is the one adopted here, as codified by the random variable $h_n$ in Lemma~\ref{lem:b}.
\end{rmk}

It is part of the setup that $X$ is a metric space and
$T:X\to X$ is the canonical billiard map corresponding to the first collision with the boundary of the
billiard table.   It is assumed (and for many classes of billiards explicitly constructed) that there is a set $Y \subset X$
and a first return map $F=T^R:Y\to Y$ such that $F$ is uniformly hyperbolic and the return time
has tail bounds satisfying
$\mu(R>n)=O(n^{-(\beta_0+1)})$, where we assume that $\beta_0>1$ (see \cite[Section~4]{ChernovZhang05}).
Moreover, $Y$ is modelled by a Young tower with exponential tails~\cite{Young98}.  
A standard argument (see for example~\cite[Theorem~4]{ChernovZhang05}) shows that
$T:X\to X$ is modelled by a Young tower $f:\Delta\to\Delta$ with polynomial tails~\cite{Young99}, with 
tail rate $O(n^{-(\beta+1)})$ for all $\beta<\beta_0$.
In particular, 
there is a measure-preserving semiconjugacy $\pi:\Delta\to X$, so we can work with $f:\Delta\to\Delta$ instead of $T:X\to X$ and observables $\hat\phi=\phi\circ\pi:\Delta\to\R$ where
$\phi:X\to\R$ is H\"older.

The final part of the set up that we require is that $\Delta$ is covered by stable leaves $\cW^s$ satisfying 
$T(W(x)) \subseteq W(Tx)$, for all $x \in \Delta$, where $W(x)$ is the element of $\cW^s$ containing $x$.
Due to the uniform hyperbolicity of $F = T^R$, the contraction condition in Lemma~\ref{lem:b} 
holds \cite[Section~4.2]{ChernovZhang05}.
Hence $f:\Delta\to\Delta$ satisfies condition~(b) of Theorem~\ref{thm:main} and it remains to verify condition~(a).

Let $\bar f:\bar\Delta\to\bar\Delta$ denote the quotient (one-sided) Young tower obtained by quotienting along stable leaves.
Consider observables $\bar\phi:\bar\Delta\to\R$ that are Lipschitz with respect to a symbolic metric on $\bar\Delta$, with Lipschitz norm $\|\bar\phi\|$.
By~\cite[Theorem~3]{Young99}, there is a constant $C>0$ such that
\begin{equation} \label{eq:Young}
\Big|\int_{\bar\Delta} \bar\phi\, \bar\psi \circ \bar f^n\, d\bar\mu_\Delta -\int_{\bar\Delta} \bar\phi\,d\bar\mu_\Delta\int_{\bar\Delta} \bar\psi\,d\bar\mu_\Delta\Big|\le C\|\bar\phi\||\bar\psi|_\infty n^{-\beta},
\end{equation}
for all $\bar\phi:\bar\Delta\to\R$ Lipschitz, $\bar\psi\in L^\infty(\bar\Delta)$, $n\ge1$.
(The dependence on $\|\bar\phi\|$ and $|\bar\psi|_\infty$ is not stated explicitly in \cite[Theorem~3]{Young99} but follows by a standard argument using the uniform boundedness principle.  Alternatively, see~\cite{KKMapp} for a direct argument.)

Returning to the two-sided tower $f:\Delta\to\Delta$ and the lifted observable
$\hat\phi=\phi\circ\pi:\Delta\to\R$, it follows for instance from~\cite[Proposition~5.3]{KKMapp} that there exists a choice (depending only on the H\"older exponent of $\phi$) of symbolic metric on $\bar\Delta$ and a sequence of observables
$\tilde\phi_\ell\in L^\infty(\Delta)$, $\ell\ge1$, such that
\begin{itemize}
\item[(i)] $\tilde\phi_\ell$ is $\cF_0$-measurable and hence projects down to an observable $\bar\phi_\ell:\bar\Delta\to\R$.  
\item[(ii)] 
$\sup_{\ell\ge1}\|L^\ell \bar\phi_\ell\|<\infty$.
\hspace{1em} (iii) $\lim_{\ell\to\infty}|\phi\circ f^\ell-\tilde\phi_\ell|_1=0$.
\end{itemize}
Here, $\cF_0$ is the $\sigma$-algebra generated by $\cW^s$ and $L$ is the transfer operator corresponding to $\bar f:\bar\Delta\to\bar\Delta$.

Let $\psi\in L^\infty(\cF_0)$ with projection $\bar\psi\in L^\infty(\bar\Delta)$.
Following~\cite[Proof of Corollary~5.4]{KKMapp}, 
\[
\SMALL
\int_{\Delta} \phi\, \psi \circ f^n\, d\mu_\Delta =
\int_{\Delta} \phi\circ f^\ell\, \psi \circ f^{\ell+n}\, d\mu_\Delta =
I_1+I_2+I_3,
\]
where
\begin{align*}
I_1 & =
\int_{\Delta} (\phi\circ f^\ell-\tilde\phi_\ell)\, \psi \circ f^{\ell+n}\, d\mu_\Delta, \qquad
 I_2  = 
\int_\Delta \tilde\phi_\ell\,d\mu_\Delta 
\int_\Delta \psi\,d\mu_\Delta,
\\ I_3 & =
\int_{\Delta} \tilde\phi_\ell\, \psi \circ f^{\ell+n}\, d\mu_\Delta 
-\int_\Delta \tilde\phi_\ell\,d\mu_\Delta 
\int_\Delta \psi\,d\mu_\Delta.
\end{align*}
Now $|I_1|\le |\phi\circ f^\ell-\tilde\phi_\ell|_1|\psi|_\infty$.
Also, $I_2=\int_\Delta (\tilde\phi_\ell-\phi\circ f^\ell)\,d\mu_\Delta
\int_\Delta \psi\,d\mu_\Delta$,
so $|I_2|\le |\phi\circ f^\ell-\tilde\phi_\ell|_1|\psi|_1$.
By (iii), $\lim_{\ell\to\infty}I_j=0$ for $j=1,2$.
By~(i),
\begin{align*}
I_3 & = 
\int_{\bar\Delta} \bar\phi_\ell\, \bar\psi \circ \bar f^{\ell+n}\, d\bar\mu_\Delta 
-\int_{\bar\Delta} \bar\phi_\ell\,d\bar\mu_\Delta 
\int_{\bar\Delta} \bar\psi\,d\bar\mu_\Delta
  \\ & 
= \int_{\bar\Delta} L^\ell \bar\phi_\ell\, \bar\psi \circ \bar f^n\, d\bar\mu_\Delta 
-\int_{\bar\Delta} L^\ell \bar\phi_\ell\,d\bar\mu_\Delta 
\int_{\bar\Delta} \bar\psi\,d\bar\mu_\Delta,
\end{align*}
so by~\eqref{eq:Young} and (ii),
$|I_3|\le C\|L^\ell\bar\phi_\ell\||\bar\psi|_\infty\, n^{-\beta}\ll |\psi|_\infty\, n^{-\beta}$.
Together, these estimates establish condition (a) in Theorem~\ref{thm:main}.


\section{Application to Lorentz gases}
\label{sec:gas}

In this section, we use the results of \cite{BaladiDemersLiverani18} to show that the hypotheses of Theorem~\ref{thm:main} (with $\beta>1$ arbitrarily large) are satisfied for the 
time-one map corresponding to a finite horizon planar periodic Lorentz gas for all H\"older observables $\phi$.
Hence the results of Section~\ref{sec:mart} hold for all $p \ge 1$, establishing the results listed in the introduction.

\subsection{Setting and main result for Lorentz gases}

Let $\T^2 = \R^2/\Z^2$ denote the two-torus, and let $B_i \subset \T^2$, $i=1, \ldots, d$, denote open convex sets 
such that their closures are pairwise disjoint and their boundaries are $C^3$ curves with strictly positive curvature.
We refer to the sets $B_i$ as scatterers.
The billiard flow $\Phi_t$ is defined by the motion of a point particle in $Q = \T^2 \setminus \bigcup_{i=1}^d B_i$ 
undergoing elastic collisions at the boundaries of the scatterers and moving at constant
velocity with unit speed between collisions.  Hence $\Phi_t$ is defined on the three dimensional phase space
\[
X = Q\times  \bbS^1, \qquad \bbS^1 = [0, 2\pi]/\sim \; , 
\]
where $\sim$ indicates that 0 and $2\pi$ are identified.

Between collisions,
$\Phi_t(x_1,x_2,\theta) = (x_1 + t\cos \theta, x_2 + t \sin \theta, \theta)$,
while at collisions the point $(x,\theta^-)$ becomes $(x,\theta^+)$ where $\theta^-$ and $\theta^+$ are the pre- and
post-collisions angles, respectively.
Defining $X_0 = X/\sim$, where we identify  $(x, \theta^-) \sim (x,\theta^+)$ at collisions, we obtain a continuous flow
$\Phi_t:X_0\to X_0$.

Let $M = \bigcup_{i=1}^d \partial B_i \times [-\pi/2, \pi/2]$.
The billiard map $F:M\to M$ is the discrete-time map which maps one collision to the next.  Parametrizing
each $\partial B_i$ by an arclength coordinate $r$ (oriented clockwise) and letting $\vf$ denote the angle that the
post-collision velocity vector makes with the normal to the scatterer (directed inwards in $Q$), we obtain the
standard coordinates $(r,\vf)$ on $M$.

For $x \in X$, define the collision time $\tau(x)$ to be the first time $t > 0$ that $\Phi_t(x) \in M$.
Since the closures of the scatterers are disjoint, there exists $\tau_{\min} >0$ such that $\tau(x) \ge \tau_{\min}$
for all $x \in M$.  In addition, we assume that the billiard has finite horizon so that there exists
$\tau_{\max} < \infty$ such that $\tau(x) \le \tau_{\max}$ for all $x \in X$. 

It is well known (see \cite[Section~3.3]{ChernovMarkarian}) that the flow preserves the contact form
\[
\omega = \cos \theta \, dx_1 + \sin \theta \, dx_2,
\]
so that the contact volume is $\omega \wedge d\omega = dx_1 \wedge d\theta \wedge dx_2$.  
We denote by $\mu$ the normalized Lebesgue measure on $X$, which by the preceding calculation is preserved by the
flow.  

The main result of this section is the following.

\begin{thm}
\label{thm:app}
Let $T$ be the time-one map corresponding to a finite horizon Lorentz gas as described above,
and let $\phi:X\to\R$ be a mean zero H\"older observable.
Then conditions (a) and (b) of Theorem~\ref{thm:main} hold 
with $n^{-\beta}$ replaced by $e^{-cn}$ for some $c>0$.

As a consequence,  conditions~\eqref{eq:mart} and~\eqref{eq:proj} hold for all $p \ge 1$, and all the results described in Section~\ref{sec:mart}
apply in this setting.
\end{thm}

We remark that the observable $\phi$ is assumed to be H\"older continuous only on $X$, not $X_0$.  Thus
$\phi$ is allowed to be discontinuous at the boundary of $X$, i.e. at collisions.  In particular, Theorem~\ref{thm:app} applies
to the velocity.


\subsection{Proof of Theorem~\ref{thm:app}}

The remainder of this section is devoted to the proof of Theorem~\ref{thm:app},
which consists of verifying the conditions of Theorem~\ref{thm:main}.
First we recall some of the essential properties and main constructions used in \cite{BaladiDemersLiverani18}.


\paragraph{Hyperbolicity and singularities}

The singularities for both the collision map and the flow are created by tangential collisions with the scatterers.
Let $\cS_0 = \{ (r,\vf) \in M : \vf = \pm\frac{\pi}{2} \}$.  Away from the set $\cS_1 = \cS_0 \cup F^{-1} \cS_0$ 
(resp. $\cS_{-1} = \cS_0 \cup F\cS_0$)
the map $F$ (resp. $F^{-1}$) is uniformly hyperbolic:  Letting 
\begin{equation} \label{eq:Lambda}
\Lambda = 1 + 2 \tau_{\min} \cK_{\min},
\end{equation}
 where
$\cK_{\min}$ denotes the minimum curvature of the scatterers, there exist stable $\bar{C}^s$ and unstable
$\bar{C}^u$ cones in the tangent space of $M$ such that stable and unstable vectors in these cones undergo
uniform expansion and contraction at an exponential rate given by $\Lambda$.  Flowing
$\bar{C}^s$ backward and $\bar{C}^u$ forward between collisions allows us to define
two families of stable $C^s$ and unstable $C^u$ cones for the flow that lie in the kernel of the contact form. (Hence they
are `flat' two-dimensional cones in the tangent space of the flow; see \cite[Sect.~2.1]{BaladiDemersLiverani18} for an explicit definition
of these cones.)

Let $P^\pm$ denote the projections from
$X$ onto $M$ under the forward and backward flow.   Then $C^u$ is continuous on $X$ away from the
surface
$\cS_{-1}^{-} = \{ x \in X : P^+(x) \in \cS_{-1} \}$, and $C^s$ is continuous on $X$ away from
the surface $\cS_1^{+} = \{ x \in X : P^-(x) \in \cS_1 \}$.
To maintain control of distortion, we define the standard homogeneity strips
\[
\mathbb{H}_k = \big\{ (r, \vf) \in M : \tfrac{\pi}{2} - \tfrac{1}{k^2} \le \vf \le \tfrac{\pi}{2} - \tfrac{1}{(k+1)^2} \big\}, \quad k\ge k_0,
\]
for some $k_0\ge1$ which is determined to ensure a one-step expansion condition.  A similar set 
of homogeneity strips $\mathbb{H}_{-k}$, $k \ge k_0$, is defined for $\vf$ near $-\frac{\pi}{2}$.

Following \cite{BaladiDemersLiverani18}, we define a set of admissible stable curves $\cA^s$ for the flow.  A $C^2$ curve $W$ belongs to 
$\cA^s$ if the tangent vector at each point of $W$ belongs to $C^s$, and $W$ has curvature bounded by $B_0$ and 
length $|W|$ bounded by $\delta_0$.  Here, $\delta_0 >0$ is chosen to satisfy a complexity bound
(see \cite[Lemma~3.8]{BaladiDemersLiverani18})  and $B_0$ is chosen large enough that the family $\cA^s$ is invariant under
$\Phi_{-t}$, $t \ge 0$ (once long pieces are subdivided according to the length $\delta_0$).  
We call $W \in \cA^s$ {\em homogeneous} if $P^+(W)$ lies in a single homogeneity strip.  

We define $\cW^s$ to be the family of maximal $C^2$ connected homogeneous stable manifolds for the flow.  
Note that $\cW^s$ forms a partition of $X$ (mod $\mu$-measure 0).
Moreover, each element of
$\cW^s$ (up to subdivision due to the length $\delta_0$) belongs to $\cA^s$.  When we define a homogeneous stable manifold 
$W \in \cW^s$, we take into account cuts introduced at the boundary of the extended singularity set, which includes the boundaries of the homogeneity strips.  Thus $P^+(\Phi_t W)$ lies in a single homogeneity strip for all $t \ge 0$.\footnote{Due to our definition
of $C^s$, if $W \in \cA^s$, then $P^+(W)$ is a stable curve for the map; and if $W \in \cW^s$, then $P^+(W)$ is a local
homogeneous stable manifold for the map.}
Let $\mathcal{F}_0$ denote the sigma algebra generated by elements of $\cW^s$.  
Since $\cW^s$ forms a partition of $X$,
it follows that $\mathcal{F}_0$ comprises countable unions of elements of $\cW^s$.


\paragraph{Norms and Banach spaces}

With the class of admissible stable curves defined, we can now describe the Banach spaces used to prove decay of correlations in~\cite{BaladiDemersLiverani18}.

Let $\alpha \in (0, \frac{1}{3}]$.  For $W \in \cA^s$, let $C^\alpha(W)$ denote the closure of $C^1$ functions in the Holder norm defined
by
\[
{|\psi|}_{C^\alpha(W)} = \sup_{x \in W} |\psi(x)| + \sup_{\substack{x, x' \in W \\ x \neq x'}} |\psi(x) - \psi(x')| \, d_W(x,x')^{-\alpha}, 
\]
where $d_W$ is arclength distance along $W$.
Define the weak norm of $\phi \in C^0(X)$ by
\[
|\phi|_w = \sup_{W \in \cA^s} \sup_{\substack{\psi \in C^\alpha(W) \\ {|\psi|}_{C^\alpha(W)} \le 1}} \int_W \phi \, \psi \, dm_W,
\]
where $m_W$ denotes arclength measure on $W$.
The weak space $\cB_w$ is defined as the completion of the set $\{ \phi \in C^0(X_0) : |\phi|_w < \infty \}$.

The strong norm $\| \phi \|_{\cB}$ is defined as in \cite[Section~2.3]{BaladiDemersLiverani18}.  The 
space $\cB$ is similarly defined as the
completion of a class of smooth functions on $X_0$ in the $\| \cdot \|_{\cB}$ norm.  
Since we do not need the precise definition of $\| \cdot \|_{\cB}$ here, we 
omit its definition; however, the following lemma summarizes some of the important properties of these spaces.

\begin{lemma}(\cite[Lemmas~3.9 and 3.10]{BaladiDemersLiverani18})
\label{lem:inclusion}
We have the inclusions
\[
C^1(X)\cap C^0(X) \subset \cB \subset \cB_w \subset (C^\alpha(X))^* ,
\]
where the first two inclusions are injective.  Moreover, $| \cdot |_w \le  \| \cdot \|_{\cB} \le C {| \cdot |}_{C^1(X)}$ 
and the unit ball of $\cB$ is compactly
embedded in $\cB_w$.
\end{lemma}

When we refer to functions $\phi \in C^0(X)$ as elements of $\cB$ or $\cB_w$, we identify $\phi$ with the measure
$\phi\, d\mu$.  With this identification, the two definitions of $\cL_t \phi$ given in the next section are reconciled.

The following lemma is central to our verification of condition~(a) in Theorem~\ref{thm:main}, and is a strengthening of 
\cite[Lemma~2.11]{BaladiDemersLiverani18}.  Let $C^\alpha(\cW^s)$ denote those functions which are in $C^\alpha(W)$ for all $W \in \cW^s$
with ${|\psi|}_{C^\alpha(\cW^s)} = \sup_{W \in \cW^s} {|\psi|}_{C^\alpha(W)}$
finite.

\begin{lemma}
\label{lem:key}
There exists $C>0$ such that for $\phi \in \cB_w$ and $\psi \in C^\alpha(\cW^s)$,
\[
|\phi(\psi)| \le C |\phi|_w {|\psi|}_{C^\alpha(\cW^s)} .
\]
\end{lemma}

Again, due to our identification, when $\phi \in C^0(X)$, we intend $\phi(\psi) = \int_X \phi \, \psi \, d\mu$.
Lemma~\ref{lem:key} is proved at the end of this section.

\paragraph{Transfer operator}

We define the transfer operator $\cL_t$, for $t \ge 0$, by $\cL_t \phi = \phi \circ \Phi_{-t}$, for $\phi \in C^0(X_0)$.
This can be extended to any element of $\cB_w$, and more generally a distribution of order $\alpha$ by
\[
\cL_t \phi (\psi) = \phi(\psi \circ \Phi_t) , \quad \mbox{for all $\psi \in C^\alpha(\cA^s)$, $\phi \in (C^\alpha(\cA^s))^*$.}
\]
By \cite[Lemma~4.9]{BaladiDemersLiverani18}, the map $(t, \phi) \mapsto \cL_t \phi$ from $[0,\infty) \times \cB$ to $\cB$ 
is jointly continuous, so $\{ \cL_t \}_{t \ge 0}$ is a semi-group of bounded operators on $\cB$.

Define the generator of the semi-group by
$Z\phi = \lim_{t \downarrow 0} \frac{\cL_t \phi - \phi}{t}$
for $\phi \in C^1(X)$.
While $Z$ is not a bounded operator on $\cB$, the strong continuity of $\cL_t$ implies that $Z$ is closed with
domain dense in $\cB$.  Indeed, by \cite[Lemma~7.5]{BaladiDemersLiverani18} the domain of $Z$ contains all $\phi\in C^2(X) \cap C^0(X_0)$ such that
$\nabla\phi\cdot \hat\eta\in C^0(X_0)$ where $\hat\eta$ denotes the flow direction, 
and there is a constant $C>0$ such that
\begin{equation}
\label{eq:Z}
\| Z \phi \|_{\cB} \le C {|\phi|}_{C^2(X)} \quad\text{for all such $\phi$} .
\end{equation}


\paragraph{Condition~(b) of Theorem~\ref{thm:main}}

Recall that $T$ and $F$ denote the time-one map for the flow and the collision map, respectively.
By the finite horizon condition, any $W \in \cW^s$ must undergo 
$k\ge \lfloor n/\tau_{\max} \rfloor$ collisions  after $n$ iterates by $T$.  
By \cite[Lemmas~3.3 and 3.4]{BaladiDemersLiverani18},
\[
\diam(T^nW) =|\Phi_nW| \le C |F^{k-1}(P^+(W))| \le C \Lambda^{-(k-1)} |P^+(W)| \le C' \Lambda^{-n/\tau_{\max}} |W|,
\]
where $\Lambda>1$ is the hyperbolicity constant defined in~\eqref{eq:Lambda}.
We have used here that the lengths of $P^+(W)$ and $W$ are bounded multiples of one another (indeed the Jacobian
of this map is $C^{\frac12}$, see \cite[Lemma~3.4]{BaladiDemersLiverani18}).

Let $\phi:X\to\R$ be $C^\eta$.
Then $\diam(\phi(T^nW))\le |\phi|_\eta\diam(T^nW)^{ \eta} \le C \Lambda^{-n \eta/\tau_{\max}}$.
Hence condition (b) holds 
with $n^{-\beta}$ replaced by $\Lambda^{-n \eta/\tau_{\max}}$.

\paragraph{Condition~(a) of Theorem~\ref{thm:main}}
By \cite[Theorem~1.4]{BaladiDemersLiverani18}, $Z$ has a spectral gap on $\cB$
and, using results of \cite{Butterley16}, 
$\cL_t$ admits the following decomposition:  There exists $\nu > 0$, a finite rank projector
$\Pi : \cB \to \cB$ and a family of bounded operators $\cP_t$ on $\cB$ satisfying $\Pi \cP_t = \cP_t \Pi = 0$, and a
matrix $\widehat{Z} : \Pi(\cB) \to \Pi(\cB)$ with eigenvalues $0,z_1,\dots,z_N\in\C$ satisfying $\Re z_j< -\nu$ for $j=1,\dots,N$, such that
\begin{equation}
\label{eq:L decomp}
\cL_t = \cP_t + e^{t \widehat{Z}} \Pi \quad \mbox{for all $t \ge 0$.}
\end{equation}
Moreover, there exists $C_\nu > 0$ such that for all $\phi$ in $\operatorname{Dom}(Z) \subset \cB$,
\begin{equation}
\label{eq:decay}
|\cP_t \phi |_w \le C_\nu e^{-\nu t} \| Z \phi \|_{\cB} \quad \mbox{for all $t \ge 0$.}
\end{equation}

Now suppose $\phi \in C^2(X) \cap C^0(X_0)$ is of mean zero and $\psi \in C^\alpha(\cW^s)$.  By \eqref{eq:L decomp},
\[
\SMALL
\int_X \phi \, \psi \circ \Phi_t \, d\mu = \int_X \cL_t \phi \, \psi \, d\mu
= \int_X \cP_t \phi \, \psi \, d\mu +
\int_X e^{t \widehat{Z}} \Pi \phi \, \psi \, d\mu .
\]
Hence by Lemma~\ref{lem:key},
\begin{equation}
\label{eq:split}
\Big|\int_X \phi \, \psi \circ \Phi_t \, d\mu\Big| 
\le C\Big\{ |\cP_t \phi|_w+ |e^{t \widehat{Z}} \Pi \phi |_w   \Big\}
{|\psi|}_{C^\alpha(\cW^s)}. 
\end{equation}

Letting $\Pi_0$ denote the projector corresponding to the simple eigenvalue $0$, we see that $\Pi_0 \phi = \int \phi \, d\mu=0$ since 
$\mu$ is the conformal probability measure with respect to $\cL_t$.  
Hence by Lemma~\ref{lem:inclusion},
\[
|e^{t \widehat{Z}} \Pi \phi |_w   
=|e^{t \widehat{Z}} (\Pi-\Pi_0) \phi |_w   
\le C  e^{- \nu t} |\phi |_w
\le C'  e^{- \nu t} {|\phi |}_{C^1(X)}. 
\]
By \eqref{eq:Z} and \eqref{eq:decay},
\[
|\cP_t \phi|_w 
\le C_\nu e^{-\nu t} \| Z \phi \|_{\cB} 
\le C' e^{-\nu t} {|\phi|}_{C^2(X)}. 
\]
Substituting these estimates in \eqref{eq:split}, 
\[
\SMALL | \int_X \phi \, \psi \circ T^n \, d\mu | 
= | \int_X \phi \, \psi \circ \Phi_n \, d\mu | 
\le C e^{-\nu n} {|\phi|}_{C^2(X)} {|\psi|}_{C^\alpha(\cW^s)} \quad \mbox{for all $n \ge 0$.}
\]
The result extends to $\phi \in C^\eta(X)$ as in \cite{BaladiDemersLiverani18} by a standard mollification argument.
(Exponential contraction persists with a rate dependent on $\eta$.)
In particular, there are constants $c,\,C>0$ such that
\[
\SMALL | \int_X \phi \, \psi \circ T^n \, d\mu | \le C e^{-cn} {|\phi|}_{C^\eta(X)} {|\psi|}_{C^\alpha(\cW^s)} \quad \mbox{for all $n \ge 0$,}
\]
for all $\phi \in C^\eta(X)$, $\psi\in C^\alpha(\cW^s)$.

Let $K(\cW^s)$ denote the set of bounded functions on $X$ that are constant on elements of $\cW^s$,
and let $| \cdot |_{C^0(\cW^s)} = \sup_{W \in \cW^s} | \cdot |_{C^0(W)}$.  Note that
these functions are $\cF_0$-measurable.
Moreover, $K(\cW^s) \subset C^\alpha(\cW^s)$ and 
${|\psi|}_{C^0(\cW^s)} = {|\psi|}_{C^\alpha(\cW^s)}$ for $\psi \in K(\cW^s)$.  Hence
\[
\SMALL | \int_X \phi \, \psi \circ T^n \, d\mu | \le C e^{-cn} {|\phi|}_{C^\eta(X)} {|\psi|}_{C^0(\cW^s)} \quad \mbox{for all $n \ge 0$,}
\]
for all $\phi \in C^\eta(X)$, $\psi\in K(\cW^s)$.

Finally, let $\phi \in C^\eta(X)$, $\psi\in L^\infty(\cF_0)$.  
Recall that $L^\infty(\cF_0)$ is the set of functions
in $L^\infty(\mu)$ which are $\cF_0$-measurable,
so there exists a pointwise representative $\psi'$
in the equivalence class of $\psi$ in $L^\infty(\mu)$ that is constant on local stable manifolds and such that $\sup|\psi'|=|\psi|_\infty$.
In particular, $\psi'\in K(\cW^s)$ with
${|\psi'|}_{C^0(\cW^s)}= |\psi|_\infty$ and
$\int_X|\psi-\psi'|\,d\mu = 0$.
Hence 
\[
\SMALL | \int_X \phi \, \psi \circ T^n \, d\mu | 
= | \int_X \phi \, \psi' \circ T^n \, d\mu | 
\le C e^{-cn} {|\phi|}_{C^\eta(X)} |\psi|_\infty \, .
\]
Hence condition~(a) holds with $n^{-\beta}$ replaced by $e^{-cn}$.


\vspace{2ex}
As promised, we end this section by proving Lemma~\ref{lem:key}.

\begin{pfof}{Lemma~\ref{lem:key}}
By density of $C^0(X_0)$ in $\cB_w$, it suffices to prove the lemma for $\phi \in C^0(X_0)$ and 
$\psi \in C^\alpha(\cW^s)$.

 The normalized Lebesgue measure $\mu$ on $X$ projects to the measure $\bmu=(2|\partial Q|)^{-1} \cos \vf \, dr d\vf$ on $M$; this is the unique smooth invariant probability measure for the billiard map $F$.
Let $\overline\cW^s$ denote the set of maximal connected homogeneous stable manifolds for $F$.  Note that
$P^+(\cW^s) = \overline\cW^s$.  Indexing elements of $\overline\cW^s$, we write $\overline\cW^s = \{ V_\xi \}_{\xi \in \Xi}$,
which defines a (mod 0) partition of $M$.
 We disintegrate $\bmu$ into conditional measures $\bmu_\xi$ on $V_\xi$, $\xi \in \Xi$, and a factor measure
 $\lambda$ on $\Xi$.  Indeed, the conditional measures are smooth on each $V_\xi$, and we can write
 \[
 d\bmu_\xi = \bar\rho_\xi \,d\bar m_\xi \,d\lambda(\xi),
 \]
where $\bar m_\xi$ is arclength measure along $V_\xi$ (in $M$), and
\begin{equation}
\label{eq:rho}
|\log \bar\rho_\xi|_{C^{\frac13}(V_\xi)} \le C, \qquad |\bar\rho_\xi|_{C^0(V_\xi)} \le C|V_\xi|^{-1},
\end{equation}
for some $C>0$ depending only on the table $Q$ (see \cite[Corollary~5.30]{ChernovMarkarian}).  
The exponent $\frac13$ comes from the definition of the homogeneity strips.
This is the standard
decomposition of $\bmu$ into a proper standard family\footnote{Standard families in \cite{ChernovMarkarian} are
standard pairs defined on local unstable manifolds, while here we use local stable manifolds.  The
decompositions of $\mu$ have equivalent properties due to the symmetry of the map $F$
under time reversal.}
(see \cite[Example~7.21]{ChernovMarkarian}).
We further subdivide $\Xi = \bigcup_{i = 1}^d \Xi_i$, where $\Xi_i$ is the index set corresponding to each component
 $M_i =  \partial B_i \times [-\frac{\pi}{2}, \frac{\pi}{2}]$ of $M$.

Write $X = \bigcup_{i=1}^d X_i$
where $X_i = \{ x \in X : P^+(x) \in M_i \}$.  
On each $X_i$, we represent Lebesgue measure as $d\mu = c \cos \vf \, dr\, d\vf\, ds$, where $c$ is a normalizing
constant, $(r,\vf)$ range over $M_i$,
and $s$ ranges from 0 to the maximum free flight time under the backwards flow, which we denote by
$t_i \le \tau_{\max}$.

Next, for each $\xi \in \Xi_i$, the flow surface $V_\xi^{-} = \{ x \in X_i : P^+(x) \in V_\xi \}$ is smoothly foliated by elements
of $\cW^s$, which are simply flow translates of one another.
For each $s$ and $V_\xi$, let $W_{\xi,s} = \Phi_{-t(s)} V_\xi$, where $t(s,z)$ is defined for $z \in V_\xi$ so that
$W_{\xi,s}$ lies in the kernel of $\omega$, i.e. it is an element of $\cW^s$.  Note that for $s < \delta_0$, some points
in $V_\xi$ may not have lifted off of $M$.  For such small times, $W_{\xi,s}$ denotes only those points that have lifted off of $M$.
Similarly, for $s > \tau_{\min}$, some part of $\Phi_{-t(s)}V_\xi$ may have collided with a scatterer.  For such times, 
$W_{\xi,s}$ only denotes those points which have not yet undergone a collision.  
Thus $\bigcup_{s \in [0, t_i]} W_{\xi,s} = V_\xi^{-}$.

Using this decomposition, we may represent Lebesgue measure on each $X_i$ by
\[
d\mu(x) = \rho_\xi(x) \, dm_{W_{\xi,s}}(x) \, d\lambda(\xi) \, ds, 
\]
where $\rho_\xi$ is smooth along each $W_{\xi,s}$, satisfying analogous bounds to \eqref{eq:rho}, since the contact
form is $C^\infty$ on $X_i$ and the projection $P^+$ is sufficiently smooth (see \cite[Lemma~3.4]{BaladiDemersLiverani18}), so that the arclength
of $W_{\xi,s}$ varies smoothly with that of $V_\xi$.   

Using the fact that each $W_{\xi,s} \in \cW^s$ can be subdivided into 
at most $C\delta_0^{-1}$ elements of $\cA^s$, we are ready to estimate
\[
\begin{split}
\Big| \int_X \phi \, \psi \, d\mu \Big| & \le \sum_{i=1}^d \Big| \int_{X_i} \phi \, \psi \, d\mu \Big| 
 \le \sum_i \Big| \int_0^{t_i} \int_{\Xi_i} \int_{W_{\xi,s}} \phi \, \psi \, \rho_\xi \, dm_{W_{\xi,s}} \, d\lambda(\xi) \, ds \Big| \\
& \le \sum_i  \int_0^{t_i} \int_{\Xi_i} C\delta_0^{-1} |\phi|_w {|\psi|}_{C^\alpha(W_{\xi,s})} {|\rho_\xi|}_{C^\alpha(W_{\xi,s})}
 \, d\lambda(\xi) \, ds \\
 & \le C\delta_0^{-1} \tau_{\max}  |\phi|_w {|\psi|}_{C^\alpha(\cW^s)} \int_{\Xi} |V_\xi|^{-1} d\lambda(\xi) \, .
\end{split}
\]
This last integral is finite by \cite[Exercise~7.15]{ChernovMarkarian} since our decomposition of $\bmu$ constitutes a
proper standard family, yielding the desired estimate for $\phi(\psi)$.

For completeness, we finish by proving \cite[Exercise~7.15]{ChernovMarkarian}.
For $x \in V_\xi$, let $r^s(x)$ denote the distance measured along $V_\xi$ from $x$ to the nearest endpoint of $V_\xi$.  By
\cite[Theorem~5.17]{ChernovMarkarian}, there exists $C_0>0$ such that 
\[
\sup_{\ve > 0}\, \frac{\bmu(x \in M : r^s(x) < \ve )}{\ve} \le C_0 \, .
\]
We claim this quantity provides an upper bound on the relevant integral.  To see this, we use the decomposition
\eqref{eq:rho} to write,
\[
\begin{split}
C_0 & = \sup_{\ve > 0} \frac{\bmu(x \in M : r^s(x) < \ve )}{\ve} = \sup_{\ve > 0 } \int_{\Xi} \frac{\bmu_\xi(r^s(x) < \ve)}{\ve} \, d\lambda(\xi) \\
& \ge \sup_{\ve > 0} \int_{\Xi} \frac{C}{|V_\xi|} \frac{|V_\xi \cap \{ r^s < \ve \}|}{\ve} \, d\lambda(\xi) \\
& \ge \sup_{\ve > 0} 2C \int_{\{\xi : |V_\xi| > 2\ve\}}  \frac{1}{|V_\xi|} \, d\lambda(\xi)
= 2C \int_{\Xi} \frac{1}{|V_\xi|} \, d\lambda(\xi) \, ,
\end{split}
\]
where we have used the fact that $|V_\xi|>0$ for $\lambda$-a.e. $\xi$, and the bound
$|V_\xi \cap \{ r^s(x) < \ve \}| = 2\ve$ if $|V_\xi| > 2\ve$.  
(One can also prove a reverse inequality, but we do not need this here.)
\end{pfof}

\paragraph{Acknowledgements}
We are very grateful to the referees for several helpful suggestions, especially with regard to the issue mentioned in Remark~\ref{rmk:error}. 

This research resulted from a Research in Pairs on {\em Techniques des martingales et espaces de Banach anisotropes (Martingale techniques  and  anisotropic Banach spaces)} at CIRM, Luminy, during August 2017.

MD was supported in part by NSF Grant DMS 1800321.
IM was supported in part by 
European Advanced Grant {\em StochExtHomog} (ERC AdG 320977).  
 MN was supported in part by NSF Grant DMS 1600780.


\begin{thebibliography}{10}

\bibitem{AMV15}
V.~Ara{\'u}jo, I.~Melbourne and P.~Varandas. Rapid mixing for the {L}orenz
  attractor and statistical limit laws for their time-1 maps. \emph{Comm. Math.
  Phys.} \textbf{340} (2015) 901--938.

\bibitem{Baladi17}
V.~Baladi. The quest for the ultimate anisotropic {B}anach space. \emph{J.
  Stat. Phys.} \textbf{166} (2017) 525--557.

\bibitem{BaladiDemersLiverani18}
V.~Baladi, M.~F. Demers and C.~Liverani. Exponential decay of correlations for
  finite horizon {S}inai billiard flows. \emph{Invent. Math.} \textbf{211}
  (2018) 39--177.

\bibitem{BalintGouezel06}
P.~B{\'a}lint and S.~Gou{\"e}zel. Limit theorems in the stadium billiard.
  \emph{Comm. Math. Phys.} \textbf{263} (2006) 461--512.

\bibitem{BM18}
P.~B{\'a}lint and I.~Melbourne. Statistical properties for flows with unbounded
  roof function, including the Lorenz attractor. \emph{J. Stat. Phys.}
  \textbf{172} (2018) 1101--1126.

\bibitem{BlankKellerLiverani}
M.~Blank, G.~Keller and C.~Liverani. Ruelle-{P}erron-{F}robenius spectrum for
  {A}nosov maps. \emph{Nonlinearity} \textbf{15} (2002) 1905--1973.

\bibitem{Bowen75}
R.~Bowen. \emph{{Equilibrium States and the Ergodic Theory of Anosov
  Diffeomorphisms}}. Lecture Notes in Math. \textbf{470}, Springer, Berlin,
  1975.

\bibitem{BunimovichSinaiChernov91}
L.~A. Bunimovich, Y.~G. Sina{\u\i} and N.~I. Chernov. Statistical properties
  of two-dimensional hyperbolic billiards. \emph{Uspekhi Mat. Nauk} \textbf{46}
  (1991) 43--92.

\bibitem{Burkholder73}
D.~L. Burkholder. Distribution function inequalities for martingales.
  \emph{Ann. Probability} \textbf{1} (1973) 19--42.

\bibitem{Butterley16}
O.~Butterley. A note on operator semigroups associated to chaotic flows.
  \emph{Ergodic Theory Dynam. Systems} \textbf{36} (2016) 1396--1408.


\bibitem{Chernov07}
N.~Chernov. {A stretched exponential bound on time correlations for billiard
  flows}. \emph{J. Stat. Phys.} \textbf{127} (2007) 21--50.

\bibitem{ChernovMarkarian}
N.~Chernov and R.~Markarian. \emph{Chaotic billiards}. Mathematical Surveys and
  Monographs \textbf{127}, American Mathematical Society, Providence, RI, 2006.

\bibitem{ChernovZhang05}
N.~Chernov and H.-K. Zhang. {Billiards with polynomial mixing rates}.
  \emph{Nonlinearity} \textbf{18} (2005) 1527--1553.

\bibitem{ChernovZhangSD}
N.~Chernov and H.-K. Zhang. {A family of chaotic billiards with variable mixing rates}.
  \emph{Stochastics and Dynamics} \textbf{5} (2005) 535--553.
  
\bibitem{CFKMZ}
I.~Chevyrev, P.~K. Friz, A.~Korepanov, I.~Melbourne and H.~Zhang. Multiscale
  systems, homogenization, and rough paths. 
\emph{Probability and Analysis in
  ``Interacting Physical Systems: In Honor of S.R.S. Varadhan, Berlin, August,
  2016''} (P.~Friz et~al., ed.), Springer Proceedings in Mathematics \&
  Statistics \textbf{283}, 2019, p.~17--48.

\bibitem{CFKMZprep}
I.~Chevyrev, P.~K. Friz, A.~Korepanov, I.~Melbourne and H.~Zhang. 
Deterministic homogenization under optimal moment assumptions for fast-slow systems.  Part~2.  Preprint, 2019.

\bibitem{DedeckerMerlevedePene13}
J.~Dedecker, F.~Merlev\`{e}de and F.~P\`{e}ne. Empirical central limit
  theorems for ergodic automorphisms of the torus. \emph{ALEA Lat. Am. J.
  Probab. Math. Stat.} \textbf{10} (2013) 731--766.

\bibitem{DedeckerRio00}
J.~Dedecker and E.~Rio. On the functional central limit theorem for stationary
  processes. \emph{Ann. Inst. H. Poincar\'e Probab. Statist.} \textbf{36}
  (2000) 1--34.

\bibitem{Demers18}
M.~F. Demers. A gentle introduction to anisotropic Banach spaces. \emph{Chaos
  Solitons Fractals} \textbf{116} (2018) 29--42.

\bibitem{Dolgopyat98b}
D.~Dolgopyat. {Prevalence of rapid mixing in hyperbolic flows}. \emph{Ergodic
  Theory Dynam. Systems} \textbf{18} (1998) 1097--1114.

\bibitem{FrizHairer}
P.~K. Friz and M.~Hairer. \emph{A course on rough paths}. Universitext,
  Springer, Cham, 2014. 

\bibitem{Gordin69}
M.~I. Gordin. {The central limit theorem for stationary processes}.
  \emph{Soviet Math. Dokl.} \textbf{10} (1969) 1174--1176.


\bibitem{Gouezel10}
S.~Gou{\"e}zel. Almost sure invariance principle for dynamical systems by
  spectral methods. \emph{Ann. Probab.} \textbf{38} (2010) 1639--1671.

\bibitem{Gouezel15}
S.~Gou\"{e}zel. Limit theorems in dynamical systems using the spectral method.
  \emph{Hyperbolic dynamics, fluctuations and large deviations}. Proc. Sympos.
  Pure Math. \textbf{89}, Amer. Math. Soc., Providence, RI, 2015, pp.~161--193.

\bibitem{HennionHerve}
H.~Hennion and L.~Herv{\'e}. \emph{{Limit Theorems for Markov Chains and
  Stochastic Properties of Dynamical Systems by Quasi-Compactness}}. Lecture
  Notes in Math. \textbf{1766}, Springer, Berlin, 2001.

\bibitem{Heyde75}
C.~C. Heyde. On the central limit theorem and iterated logarithm law for
  stationary processes. \emph{Bull. Austral. Math. Soc.} \textbf{12} (1975)
  1--8.

\bibitem{KM16}
D.~Kelly and I.~Melbourne. {Smooth approximation of stochastic differential
  equations}. \emph{Ann. Probab.} \textbf{44} (2016) 479--520.

\bibitem{KM17}
D.~Kelly and I.~Melbourne. Homogenization for deterministic fast-slow systems
  with multidimensional multiplicative noise. \emph{J. Funct. Anal.}
  \textbf{272} (2017) 4063--4102.

\bibitem{KipnisVaradhan86}
C.~Kipnis and S.~R.~S. Varadhan. Central limit theorem for additive functionals
  of reversible {M}arkov processes and applications to simple exclusions.
  \emph{Comm. Math. Phys.} \textbf{104} (1986) 1--19.

\bibitem{KKMapp}
A.~Korepanov, Z.~Kosloff and I.~Melbourne. Explicit coupling argument for
  nonuniformly hyperbolic transformations. \emph{Proc. Roy. Soc. Edinburgh A}
\textbf{149} (2019) 101--130.

\bibitem{Liverani96}
C.~Liverani. {Central limit theorem for deterministic systems}.
  \emph{{International Conference on Dynamical Systems}} (F.~Ledrappier,
  J.~Lewowicz and S.~Newhouse, eds.), Pitman Research Notes in Math.
  \textbf{362}, Longman Group Ltd, Harlow, 1996, pp.~56--75.

\bibitem{Lyons98}
T.~J. Lyons. Differential equations driven by rough signals. \emph{Rev. Mat.
  Iberoamericana} \textbf{14} (1998) 215--310.

\bibitem{Markarian04}
R.~Markarian. Billiards with polynomial decay of correlations. \emph{Ergodic
  Theory Dynam. Systems} \textbf{24} (2004) 177--197.

\bibitem{MaxwellWoodroofe00}
M.~Maxwell and M.~Woodroofe. Central limit theorems for additive functionals of
  {M}arkov chains. \emph{Ann. Probab.} \textbf{28} (2000) 713--724.

\bibitem{M07}
I.~Melbourne. {Rapid decay of correlations for nonuniformly hyperbolic flows}.
  \emph{Trans. Amer. Math. Soc.} \textbf{359} (2007) 2421--2441.

\bibitem{M18}
I.~Melbourne. Superpolynomial and polynomial mixing for semiflows and flows.
  \emph{Nonlinearity} \textbf{31} (2018) R268--R316.

\bibitem{MN05}
I.~Melbourne and M.~Nicol. Almost sure invariance principle for nonuniformly
  hyperbolic systems. \emph{Comm. Math. Phys.} \textbf{260} (2005) 131--146.

\bibitem{MN08}
I.~Melbourne and M.~Nicol. Large deviations for nonuniformly hyperbolic
  systems. \emph{Trans. Amer. Math. Soc.} \textbf{360} (2008) 6661--6676.


\bibitem{MT02}
I.~Melbourne and A.~T{\" o}r{\" o}k. {Central limit theorems and invariance
  principles for time-one maps of hyperbolic flows}. \emph{Comm. Math. Phys.}
  \textbf{229} (2002) 57--71.

\bibitem{MT04}
I.~Melbourne and A.~T{\" o}r{\" o}k. {Statistical limit theorems for suspension
  flows}. \emph{Israel J. Math.} \textbf{144} (2004) 191--209.

\bibitem{MTorok12}
I.~Melbourne and A.~T{\" o}r{\" o}k. {Convergence of moments for Axiom A and
  nonuniformly hyperbolic flows}. \emph{Ergodic Theory Dynam. Systems}
  \textbf{32} (2012) 1091--1100.

\bibitem{MV16}
I.~Melbourne and P.~Varandas. A note on statistical properties for nonuniformly
  hyperbolic systems with slow contraction and expansion. \emph{Stoch. Dyn.}
  \textbf{16} (2016) 1660012, 13 pages.


\bibitem{PeligradUtev05}
M.~Peligrad and S.~Utev. A new maximal inequality and invariance principle for
  stationary sequences. \emph{Ann. Probab.} \textbf{33} (2005) 798--815.

\bibitem{Ruelle78}
D.~Ruelle. \emph{{Thermodynamic Formalism}}. Encyclopedia of Math. and its
  Applications \textbf{5}, Addison Wesley, Massachusetts, 1978.


\bibitem{Sinai72}
Y.~G. Sina{\u\i}. {Gibbs measures in ergodic theory}. \emph{Russ. Math. Surv.}
  \textbf{27} (1972) 21--70.

\bibitem{TyranKaminska05}
M.~Tyran-Kami{\'n}ska. An invariance principle for maps with polynomial decay
  of correlations. \emph{Comm. Math. Phys.} \textbf{260} (2005) 1--15.

\bibitem{Viana}
M.~Viana. \emph{{Stochastic dynamics of deterministic systems}}. Col. Bras. de
  Matem{\'a}tica, 1997.

\bibitem{Volny93}
D.~Voln\'y. Approximating martingales and the central limit theorem for
  strictly stationary processes. \emph{Stochastic Process. Appl.} \textbf{44}
  (1993) 41--74.

\bibitem{Volny07}
D.~Voln\'{y}. A nonadapted version of the invariance principle of {P}eligrad
  and {U}tev. \emph{C. R. Math. Acad. Sci. Paris} \textbf{345} (2007)
  167--169.

\bibitem{Young98}
L.-S. Young. Statistical properties of dynamical systems with some
  hyperbolicity. \emph{Ann. of Math.} \textbf{147} (1998) 585--650.

\bibitem{Young99}
L.-S. Young. Recurrence times and rates of mixing. \emph{Israel J. Math.}
  \textbf{110} (1999) 153--188.

\end{thebibliography}
\end{document}